\documentclass[10pt]{article}
\usepackage{amsmath}
\usepackage{amsthm}
\usepackage{amssymb}
\usepackage{tabu}

\usepackage{bbm}

\newcommand{\LO}{\ensuremath{L_{\gamma,\alpha}}}
\newcommand{\LOl}{\ensuremath{L_{\gamma,2}}}
\newcommand{\Bm}{\ensuremath{\beta_-(\gamma,\alpha)}}
\newcommand{\Bp}{\ensuremath{\beta_+(\gamma,\alpha)}}

\newcommand{\R}{\ensuremath{\mathbb{R}}}
\newcommand{\h}{\ensuremath{H^{\frac{\alpha}{2}}_0(\mathbb{R}^n)}}
\newcommand{\homega}{\ensuremath{H_0^{\frac{\alpha}{2}}(\Omega)}}
\newcommand{\homegal}{\ensuremath{H_0^1(\Omega)}}
\newcommand{\N}{\ensuremath{\mathbb{N}}}

\def \rr {\mathbb{R}}

\def \rn {\mathbb{R}^n}
\def \ue {u_\epsilon}
\def \Te {T_{\epsilon,\alpha}}
\def \eps {\epsilon}
\def \crits {2_{\alpha}^*(s)}
\def \critsl {2_{2}^*(s)}
\def \bp {\beta_+(\alpha)}
\def \bm {\beta_-(\alpha)}
\def \bpl {\beta_+(2)}
\def \bml {\beta_-(2)}

\def \mur {\mu_{\gamma,s, \alpha}(\R^n)} 
\def \murl {\mu_{\gamma,s, 2}(\R^n)} 
\def \fo {(-\Delta)^{\frac{\alpha}{2}}}

\newcommand{\ir}{\int_{\R ^n}}

\newcommand{\io}{\int_{\Omega}}
\newcommand{\bs}{\begin{split}}
\newcommand{\es}{\end{split}}

\title{Existence Result for Non-linearly Perturbed Hardy-Schr\"odinger Problems: Local and Non-local cases}

\author{Shaya Shakerian  \\
{\it\small Department of Mathematics}\\
{\it\small  University of British Columbia}\\
{\it\small Vancouver BC Canada V6T 1Z2}\\
{\it\small shaya@math.ubc.ca}\vspace{1mm}
\\\\
}

\newtheorem{definition}{Definition}[section]
\newtheorem{theorem}[definition]{Theorem}
\newtheorem{lemma}[definition]{Lemma}
\newtheorem{proposition}[definition]{Proposition}

\newtheorem{remark}[definition]{Remark}
\theoremstyle{definition}

\begin{document}

 {\date{}}
\maketitle

\begin{abstract}
Let $\Omega \subset \R^n$ be a smooth bounded domain having zero in its interior $0 \in \Omega.$  We fix $0 < \alpha  \le 2$  and $0 \le s <\alpha.$ We investigate a sufficient condition for the existence of a positive solution for  the following perturbed problem associated with the Hardy-Schr\"odinger operator   $ \LO: = ({-}{ \Delta})^{\frac{\alpha}{2}}-  \frac{\gamma}{|x|^{\alpha}}$ on $\Omega:$ 
\begin{equation*}\label{Main problem: M-P solutions:Abstract}
\left\{\begin{array}{rl}
\displaystyle ({-}{ \Delta})^{\frac{\alpha}{2}}u- \gamma \frac{u}{|x|^{\alpha}} - \lambda u= {\frac{u^{\crits-1}}{|x|^s}}+ h(x) u^{q-1} & \text{in }  {\Omega}\vspace{0.1cm}\\

 u=0 \,\,\,\,\,\,\,\,\,\,\,\,\,\,\,\,\,\,\,\,\,\,\,\,\,\,\,\,\,\,\,\,\,\,\,\,\,\,\,\,\,\,\,\,\,\,\,\,\,\,\,\,  & \text{in } \R^n \setminus \Omega,
\end{array}\right.
\end{equation*}
where  ${2_{\alpha}^*(s)}:=\frac{2(n-s)}{n-{\alpha}},$ $\lambda \in \R $,  $h \in C^0(\overline{\Omega}),$ $h \ge 0,$ $q \in (2, 2^*_\alpha)$ with $2^*_\alpha:=2^*_\alpha(0),$ and  $\gamma < \gamma_H(\alpha),$ the latter being the best constant in the Hardy inequality on $\R^n.$   We prove that there exists a threshold $  \gamma_{crit}(\alpha)$ in $( - \infty, \gamma_H(\alpha)) $ such that the existence of solutions of the above problem is guaranteed by the non-linear perturbation $(i.e., h(x) u^{q-1})$ whenever $ \gamma \le \gamma_{crit}(\alpha),$ while for $\gamma_{crit}(\alpha)<\gamma <\gamma_H(\alpha)$, it is determined by a subtle combination of  the geometry of the domain and the size of the nonlinearity of the perturbations.


\end{abstract}

\section{Introduction}

Let $h(x)$ be a  non-negative function  in $ C^0(\overline{\Omega}).$ Given  $0< \alpha \le 2$ and $ 0 \le s < \alpha,$ we consider the following perturbed problem associated with the operator $ \LO: = ({-}{ \Delta})^{\frac{\alpha}{2}}-  \frac{\gamma}{|x|^{\alpha}}$ on bounded domains $\Omega \subset \R^n (n>\alpha)$ with $0 \in \Omega$:
\begin{equation}\label{Main problem: M-P solutions}
\left\{\begin{array}{rl}
\displaystyle ({-}{ \Delta})^{\frac{\alpha}{2}}u- \gamma \frac{u}{|x|^{\alpha}} - \lambda u= {\frac{u^{\crits-1}}{|x|^s}}+ h u^{q-1} & \text{in }  {\Omega}\vspace{0.1cm}\\
u\geq 0  \,\,\,\,\,\,\,\,\,\,\,\,\,\,\,\,\,\,\,\,\,\,\,\,\,\,\,\,\,\,\,\,\,\,\,\,\,\,\,\,\,\,  & \text{in } \Omega ,\\
 u=0 \,\,\,\,\,\,\,\,\,\,\,\,\,\,\,\,\,\,\,\,\,\,\,\,\,\,\,\,\,\,\,\,\,\,\,\,\,\,\,\,\,\,  & \text{in } \R^n \setminus \Omega,
\end{array}\right.
\end{equation}
where ${2_{\alpha}^*(s)}:=\frac{2(n-s)}{n-{\alpha}},$  $q \in (2, 2^*_\alpha)$ with $2_{\alpha}^*:=2_{\alpha}^*(0),$ $\lambda \in \R ,$ and $\gamma < \gamma_H(\alpha):=2^\alpha \frac{\Gamma^2(\frac{n+\alpha}{4})}{\Gamma^2(\frac{n-\alpha}{4})},$ the latter being the best constant in the Hardy inequality on $\R^n$ defined in (\ref{fractional Trace Hardy inequality}). The structure of the operator $\fo$ has broad influence on problem (\ref{Main problem: M-P solutions}). More precisely, if $\alpha =2,$ this operator is defined as 
$$(-\Delta)^{\frac{\alpha}{2}}u:= -\Delta u = \sum_{i=0}^{n} \frac{\partial^2 u }{\partial x_i^2} ,
  \quad  \text{ for } u \in C^2.$$ It is well-known that  the classical Laplacian  $-\Delta$ is a local operator. Therefore, problem (\ref{Main problem: M-P solutions}) involves a local equation when $\alpha =2.$ While, for $0< \alpha<2,$ the operator $\fo$ has more complicated structure.  Indeed, for these values of $\alpha,$ the operator  is  non-local  and it is defined as 
$$
 (-\Delta)^{\frac{\alpha}{2}}u:= \mathcal{F}^{-1}(|2 \pi\xi|^{\alpha}(\mathcal{F}u)) \quad \forall\xi\in\R^n,
  \quad \text{ for } u \in \mathcal{S},$$
where  $\mathcal{S}$ is the Schwartz class (space of rapidly decaying $C^\infty$ functions
in $\R^n$) and $ \mathcal{F}u$ denotes the Fourier transform of $u.$ Hence, the operator $\fo$ and problem (\ref{Main problem: M-P solutions})  are non-local when $0 < \alpha <2.$\\ 



When $(h \equiv 0)$, problem (\ref{Main problem: M-P solutions}) has been studied in the both  local and non-local cases. See \cite{Ghoussoub-Robert-2015} and \cite{GFSZ}, and the references therein. In \cite{Jaber}, Jaber considered the local problem ($i.e.,  \alpha =2$) in the Riemannian context but in the absence of the Hardy term $(i.e., \gamma =0).$ 
A similar problem, with the second order operator replaced by the fourth order Paneitz operator, was studied by Esposito-Robert \cite{Esposito-Robert} (see also Djadli- Hebey-Ledoux \cite{Djadli-Hebey- Ledoux}).  The author in \cite{Shakerian} addressed questions regarding  the existence and multiplicity of solutions of  problem (\ref{Main problem: M-P solutions}) in the case when $\lambda =0$ and  $1<q<2$ (i.e., the Concave-Convex non-linearity).  In this paper, we consider the remaining cases. 

By using ideas from \cite{GFSZ} and \cite{Jaber}, we investigate the role of the linear perturbation $(i.e., \ \lambda u),$  the non-linear  perturbation $(i.e., \ h u^{q-1})$, and the geometry of the domain 
on the existence of a positive solution of (\ref{Main problem: M-P solutions}). As in Jaber \cite{Jaber}, 
our main tool here to investigate the existence of solutions is  the  following Mountain Pass Lemma of Ambrosetti- Rabinowitz \cite{Ambrosetti-Rabinowitz}:

\begin{lemma}[Ambrosetti and Rabinowitz \cite{Ambrosetti-Rabinowitz}] \label{Theorem MPT- Ambrosetti-Rabinowitz version}
Let $(V,\|\,\|$) be a Banach space and $\Psi: {V \to \R}$ a $C^1-$functional satisfying the following conditions:\\
(a) $ \Psi(0)=0,$ \\
(b) There exist $\rho, R>0$ such that $\Psi(u) \ge \rho$ for all $u \in V$, with $\|u\|=R,$\\
(c) There exists $v_0 \in V $ such that $\limsup\limits_{t \to \infty} \Psi(tv_0) <0.$\\
Let $t_0>0$ be such that $\|t_0v_0\|> R$ and $\Psi(t_0v_0)<0,$ and define
$$c_{v_0}(\Psi):=\inf\limits_{\sigma \in \Gamma} \sup\limits_{t \in[0,1]} \Psi(\sigma(t)),$$ where  $$\Gamma:= \{\sigma\in C([0,1],V): \sigma(0)=0 \text{ and } \sigma(1)=t_0v_0 \}.$$
Then, $c_{v_0}(\Psi) \ge \rho>0$, and there exists a Palais-Smale sequence at level $c_{v_0}(\Psi)$, that is there exists a sequence  $(w_k)_{k \in \mathbb{N}} \in V$ such that 
 $$
 \lim\limits_{k \to \infty} \Psi(w_k)=c_{v_0}(\Psi) \ \text{ and } \lim\limits_{k \to \infty}  \Psi'(w_k)=0 \, \, \quad \text{strongly in} \, V' .$$
\end{lemma}

Inspired by the work of Jannelli \cite{Jannelli}, 
it was shown in \cite{GFSZ} that 
 the behaviour of problem (\ref{Main problem: M-P solutions})  is deeply influenced by the value of the parameter $\gamma$. 
More precisely,  there exists a threshold  $ \gamma_{crit}(\alpha) \in (-\infty, \gamma_H(\alpha))$ such that the
operator $\LO$
 becomes critical in the following sense:

\begin{definition}
We say that the Hardy-Schr\"odinger operator  $\LO$ is critical, \\

$\bullet$ for $0<\alpha<2,$ if $ \gamma_{crit}(\alpha)< \gamma < \gamma_H(\alpha)$  when $n \ge 2 \alpha$, or $0 \le  \gamma < \gamma_H(\alpha) $
when $\alpha < n < 2\alpha.$

$\bullet$ for $\alpha = 2,$ if $ \gamma_{crit}(2)< \gamma < \gamma_H(2)$  when $n \ge 4$, or $ \gamma < \gamma_H(2) $
when $n =3.$\\

Otherwise, the operator $\LO$ is called non-critical.
\end{definition}

 Our analysis shows that the existence of a solution for  problem (\ref{Main problem: M-P solutions}) depends only on the non-linear perturbation when the operator $\LO$ is non-critical, while the critical case is more complicated and  depends on other conditions involving both the perturbation and the global geometry of the domain. More precisely, in the non-critical case, the competition is between the linear and non-linear perturbations, and  since $q > 2,$ the non-linear term dominates. In the critical case, this competition is more challenging as it is between the  geometry of the domain (i.e., \textit{the mass}) and the non-linear perturbation. In this situation,  
there exists a threshold $q_{crit}(\alpha) \in (2, 2^*_\alpha),$ where the dominant factor switches from the non-linear perturbation to the mass. The transition at  $2^*_\alpha$ is most interesting.   
We shall establish the following result.

\begin{theorem}\label{Thm M-P: Main result }
Let $\Omega$ be a smooth bounded domain in $\R^n (n$ $ > \alpha)$ such that $0 \in \Omega,$ and let $\crits:= \frac{2(n-s)}{n-\alpha}, $  $0 \le  s < \alpha,$ $0 < \alpha \le 2,$ $ -\infty < \lambda < \lambda_1(\LO),$ and $ \gamma < \gamma_H(\alpha).$ We also assume that $2 < q <2^*_\alpha,$ $h \in C^0(\overline{\Omega})$ and $h \ge0.$ Then, there exists a non-negative  solution $u \in \homega$  to (\ref{Main problem: M-P solutions})  under one of the following conditions:\\

\smallskip\noindent (1)  $\LO$ is not critical and   $h(0)>0.$\\

(2)  $\LO$ is critical and $
\left\{\begin{array}{rl}
\displaystyle h(0)>0\,\,\,\,\,\,\,\,\,\,\,\,\,\,\,\,\,\,\,\,\,\,\,\,\,\,\,\,\,\,\,\,\,\,\,\,\,\,\,\, & \text{if }  q > q_{crit}(\alpha)\ \vspace{0.1cm}\\
c_1 h(0) + c_2 m^{\lambda}_{\gamma,\alpha} (\Omega)>0 & \text{if } q = q_{crit}(\alpha)\\
m^{\lambda}_{\gamma,\alpha}(\Omega) >0 \,\,\,\,\,\,\,\,\,\,\,\,\,\,\,\,\,\,\,\,\,\,\,\,\,\,\,\,\,\,  & \text{if } q < q_{crit}(\alpha).
\end{array}\right.
$\\

Here $c_1,c_2 $ are two positive constants that can be computed explicitly (see Subsection \ref{Subsection:estimates for non-local case} ), while $q_{crit}(\alpha) = 2^*_\alpha - 2 \frac{\Bp-\Bm}{n-\alpha} \in (2 ,2^*_\alpha).$ Also,  $\Bm$ (resp., $\Bp$) is the unique solution in 
$\left(0,\frac{n-\alpha}{2}\right)$ (resp., in $\left(\frac{n-\alpha}{2},n-\alpha\right))$ of the equation 

$$\Psi_{n,\alpha}(t):= 2^\alpha \frac{\Gamma(\frac{n-t}{2})\Gamma(\frac{\alpha+t}{2})}{  \Gamma(\frac{n-t-\alpha}{2})\Gamma(\frac{t}{2} ) }= \gamma.$$

\end{theorem}



 One can then complete the picture as follows.

\begin{center}
Non-linearly perturbed problem (\ref{Main problem: M-P solutions}): with $0 \in \Omega:$ $ 2 < q < 2^*_\alpha$ and $ \lambda < \lambda(\LO)$

\begin{tabular}{ | m{2.3cm} | m{1.9 cm}| m{1.7cm} | m{1.1cm} | m{3.5cm}| m{0.5cm}|}
\hline
  Operator $\LO$ & Singularity & \qquad  $q$ & \quad $\lambda$ & \quad Analytic. cond & Ext.\\
\hline
  Not Critical & \quad $s \ge 0  $ & \ \quad $> 2$ & $> - \infty $ &  \qquad \ $h(0) >0$ & Yes \\ 
\hline
\quad Critical  & \quad $s \ge 0  $ & $  > q_{crit}(\alpha)$& $> - \infty$ & \qquad  \ $h(0) >0$ & Yes \\
\hline
\quad Critical & \quad $s \ge 0  $  &   $  = q_{crit}(\alpha)$ &$> -\infty$ & $c_1 h(0) + c_2 m^{\lambda}_{\gamma,\alpha} (\Omega) >0$  & Yes  \\ 
\hline
\quad Critical   & \quad $s \ge 0  $  & $  < q_{crit}(\alpha)$ & \quad $>0$ &  \qquad $m^{\lambda}_{\gamma,\alpha} (\Omega)>0$ & Yes \\  
\hline
\end{tabular}
\end{center}

Theorem \ref{Thm M-P: Main result } suggests the following remarks:

\begin{remark}
The notation $m^{\lambda}_{\gamma,\alpha} (\Omega)$ stands for the mass associated to the operator $\LO - \lambda I $ which is defined in Theorem 1.2 in \cite{GFSZ} for $0 < \alpha <2,$ and in Proposition 3 in \cite{Ghoussoub-Robert-2015} for $\alpha=2.$
\end{remark}

\begin{remark}
When $\alpha =2,$ problem (\ref{Main problem: M-P solutions}) becomes local, and it can be written as follows

\begin{equation}\label{Main problem: M-P solutions-Local}
\left\{\begin{array}{rl}
\displaystyle {-}{ \Delta} u- \gamma \frac{u}{|x|^2} - \lambda u= {\frac{u^{\critsl-1}}{|x|^s}}+ h(x) u^{q-1} & \text{in }  {\Omega}\vspace{0.1cm}\\
u\geq 0  \,\,\,\,\,\,\,\,\,\,\,\,\,\,\,\,\,\,\,\,\,\,\,\,\,\,\,\,\,\,\,\,\,\,\,\,\,\,\,\,\,\,\,\,\,\,\,\,\,\,\,\,  & \text{in } \Omega ,\\
 u=0 \,\,\,\,\,\,\,\,\,\,\,\,\,\,\,\,\,\,\,\,\,\,\,\,\,\,\,\,\,\,\,\,\,\,\,\,\,\,\,\,\,\,\,\,\,\,\,\,\,\,\,\,  & \text{in } \partial \Omega,
\end{array}\right.
\end{equation}

where  $u$ belongs to the space $H_0^{1}(\Omega),$ which  is the completion of $C_c^\infty(\Omega)$ with respect to the norm $$ \|u\|_{H_0^{1}(\Omega)}^2 = \int_{\Omega} |\nabla u|^2dx.$$

\end{remark}


\begin{remark}
Note that the best Hardy constant ($\gamma_H(\alpha)$) and  the critical threshold $\gamma_{crit}(\alpha)$ can be  computed explicitly when $\alpha=2;$ see \cite{Ghoussoub-Robert-2015}. Indeed, we have  $$\gamma_{crit}(2):= \frac{(n-2)^2}{4} -1 \quad  \text{ and }  \quad  \gamma_H(2):= \frac{(n-2)^2}{4}.$$ 
\end{remark}

\begin{remark}
We point out that the value $q_{crit}(\alpha)$  corresponds to the value $q = 4$ obtained in \cite[Proposition 3]{Jaber}.  Indeed,  when $\alpha =2, $ $\gamma = 0$ and $n =3,$  our problem turns to the perturbed Hardy-Sobolev equation considered by Jaber \cite{Jaber} in the Riemannian setting. We then have that $\beta_+(2,0) = n-\alpha = 1,$ $\beta_-(2,0) =0, $ and  therefore 
$$q_{crit}(2) = 2^*_\alpha - 2 \frac{\beta_+(2,0)-\beta_-(2,0)}{n- \alpha} = 6 - 2 = 4.$$

\end{remark}

\section{ The non-local case}\label{Sec:non-local case}

Throughout this section, we shall assume that $ 0<\alpha < 2,$
which means $\fo$ is not a local operator. We start by recalling and introducing suitable function spaces for the variational principles that will be needed in the sequel. We shall study problems on bounded domains, but will start by recalling the properties of  $\displaystyle (-\Delta)^{\frac{\alpha}{2}}$ on the whole of $\R^n$, where  it can be defined on the 
Schwartz class $\mathcal{S}$ (the space of rapidly decaying $C^\infty$ functions on $\R^n$) via the Fourier transform,
\begin{equation*}
 (-\Delta)^{\frac{\alpha}{2}}u= \mathcal{F}^{-1}(| 2 \pi \xi|^{\alpha}\mathcal{F}(u)).
 \end{equation*}
Here, $ \mathcal{F}(u)$ is the Fourier transform of $u$, $\displaystyle \mathcal{F}(u)(\xi)=\int_{\R^n} e^{-2\pi i x.\xi} u(x) dx$. See Servadei-Valdinoci \cite{Hitchhikers guide} and references therein for the basics on the fractional Laplacian. For $\alpha \in (0,2)$, the fractional Sobolev space $\h$ is defined as the completion of $C_c^{\infty}(\R^n)$ under the norm 
$$\|u\|_{\h}^2= \int_{\mathbb{R}^n}|2\pi \xi |^{\alpha} |\mathcal{F}u(\xi)|^2 d\xi =\int_{\mathbb{R}^n} |(-\Delta)^{\frac{\alpha}{4}}u|^2 dx.$$
By Proposition 3.6 in Di Nezza-Palatucci-Valdinoci \cite{Hitchhikers guide} (see also Frank-Lieb-Seiringer \cite{Frank-Lieb-Seiringer}), the following relation holds: 
For $u \in \h,$ 
\begin{equation*}
 \int_{\mathbb{R}^n}|2\pi \xi |^{\alpha} |\mathcal{F}u(\xi)|^2 d\xi=\frac{C_{n,\alpha}}{2}\int_{\R^n}  \int_{\R^n} \frac{|u(x)- u(y)|^2}{|x-y|^{n+\alpha}} dxdy, \quad  \text{ where } C_{n,\alpha}=\frac{2^\alpha\Gamma\left(\frac{n+\alpha}{2}\right)}{\pi^{\frac{n}{2}}\left|\Gamma\left(-\frac{\alpha}{2}\right)\right|}. 
 \end{equation*}

The fractional Hardy inequality in $\R^n$ then states that 
\begin{equation}\label{fractional Trace Hardy inequality}
\gamma_H(\alpha):= \inf\left\{\frac{ \int_{\R^n} |({-}{ \Delta})^{\frac{\alpha}{4}}u|^2 dx  }{\int_{\R^n} \frac{|u|^2}{|x|^{\alpha}}dx};\, u \in \h \setminus \{0\}\right\}=2^\alpha \frac{\Gamma^2(\frac{n+\alpha}{4})}{\Gamma^2(\frac{n-\alpha}{4})},
\end{equation}
which means that the fractional Hardy-Schr\"odinger operator $ \LO$ is positive whenever $n >\alpha$ and $\gamma < \gamma_H(\alpha)$. 
The best  constant in Hardy-Sobolev type inequalities on $\R^n$ defined as 
\begin{equation} \label{def:mu}
\mu_{\gamma,s,\alpha}(\R^n):= \inf\limits_{u \in \h \setminus \{0\}} \frac{ \int_{\R^n} |({-}{ \Delta})^{\frac{\alpha}{4}}u|^2 dx - \gamma \int_{\R^n} \frac{|u|^2}{|x|^{\alpha}}dx }{(\int_{\R^n} \frac{|u|^{\crits}}{|x|^{s}}dx)^\frac{2}{\crits}}.
\end{equation}
We shall use the extremal of  $\mu_{\gamma,s,\alpha}(\R^n)$ and its profile at zero and infinity to build appropriate  test-functions for the functional under study.
Note that any minimizer for (\ref{def:mu}) leads --up to a constant-- to a  
variational solution of the following borderline problem  on $\R^n$,
 \begin{equation}\label{Main:problem}
\left\{\begin{array}{rl}
({-}{ \Delta})^{\frac{\alpha}{2}}u- \gamma \frac{u}{|x|^{\alpha}}= {\frac{u^{\crits-1}}{|x|^s}} & \text{in }  {\R^n}\\
 u\geq 0\; ;\; u\not\equiv 0 \,\,\,\,\,\,\,\,\,\,\,\,\,\,\,\,\,  & \text{in }  \mathbb{R}^n.
\end{array}\right.
\end{equation}

 Let now $\Omega$ be a smooth bounded domain in $\R^n$ with $0$ in its interior. We then consider the fractional Sobolev sapce $\homega$ as  the closure of $C^\infty_0(\Omega)$ with respect to the norm 
$$\|u\|^2_{\homega} =  \frac{c_{n,\alpha}}{2}\int_{\R^n}  \int_{\R^n} \frac{|u(x)- u(y)|^2}{|x-y|^{n+\alpha}} dxdy.$$ 


Note first that inequality (\ref{fractional Trace Hardy inequality}) asserts that $\homega$ is embedded in the weighted space $L^2(\Omega, |x|^{-\alpha})$ and that this embeding is continuous. If $\gamma < \gamma_H(\alpha),$ it follows from (\ref{fractional Trace Hardy inequality}) that 
$$\|| u\|| = \left( \frac{C_{n,\alpha}}{2}\int_{\rn} \int_{\rn}\frac{|u(x)-u(y)|^2}{|x-y|^{n+\alpha}}\, dxdy- \gamma \int_{\Omega}\frac{u^2}{|x|^\alpha} \ dx \right)^{\frac{1}{2}} $$
is well-defined on $\homega$. Since $0 \le \gamma < \gamma_H(\alpha),$ the following inequalities then hold for any $u \in \homega,$  

\begin{equation}\label{comparable norms}
(1-\frac{\gamma}{\gamma_H(\alpha)}) \|u\|^2_{\homega} \le \|| u\||^2 \le (1+\frac{\gamma}{\gamma_H(\alpha)}) \|u\|^2_{\homega} .
\end{equation}
Thus, $\|| \ . \ \||$ is equivalent to the norm $\| \ . \ \|_{\homega}$.

We shall consider the following functional $\Phi : \homega \to \R$  whose critical points are solutions for (\ref{Main problem: M-P solutions}): $\text{ For } u \in \homega$, let
\begin{equation*}\label{Functional Phi: for mountain pass solutions}
\Phi(u)= \frac{1}{2} \|| u\||^2 -\frac{\lambda}{2} \io u^2 dx  -\frac{1}{2_{\alpha}^*(s)}\int_{\Omega} \frac{u_+^{2_{\alpha}^*(s)}}{|x|^{s}} dx -\frac{1}{q} \int_{\Omega} h u_+^q dx,
\end{equation*}
where  $u_+ = \max(0,u)$ is the non-negative part of $u.$  Note that any critical point 
of the functional $\Phi(u)$ is  essentially  a  variational solution of (\ref{Main problem: M-P solutions}). Indeed, we have for any $ v \in \homega,$ 
\begin{align*}
\langle \Phi'(u), v \rangle &= \frac{C_{n,\alpha}}{2}\int_{(\R^n)^2} \frac{(u(x)-u(y))(v(x)-v(y))}{|x-y|^{n+\alpha}} dxdy  - \int_{\mathbb{R}^n}(\frac{\gamma u}{|x|^\alpha}+ \lambda u+ \frac{u_+^{2_\alpha^*(s)-1}}{|x|^s} + h u_+^{q-1} ) v \ dx.
\end{align*}

\subsection{General condition of existence }

In this section, we investigate a general condition of existence of solution for (\ref{Main problem: M-P solutions}).


\begin{theorem}\label{Thm M-P: main condition of existence}
Let $\Omega$ be a smooth bounded domain in $\R^n (n > \alpha)$ such that $0 \in \Omega,$ and let $\crits:= \frac{2(n-s)}{n-\alpha}, $  $0 \le  s < \alpha,$ $-\infty < \lambda < \lambda_1(\LO),$ and $0 \le \gamma < \gamma_H(\alpha).$ We consider $2 < q <2^*_\alpha,$ $h \in C^0(\overline{\Omega})$ and $h \ge0.$ . We also assume that there exists $w \in \homega$, $w \not \equiv 0$ and $w \ge 0 $ such that 
\begin{equation}\label{Formula M-P: general cond. for existence}
\sup\limits_{t \ge 0} \Phi(t w) <  \frac{\alpha-s}{2(n-s)} \mur^{\frac{n-s}{\alpha-s}}.  
\end{equation} 
Then, problem (\ref{Main problem: M-P solutions}) has a non-negative  solution in $\homega$.
\end{theorem}

We split the proof in three parts:

\subsubsection{The Palais-Smale condition below a critical threshold}


In this subsection, we prove the following


\begin{proposition}\label{Prop M-P: existence of strong limit for all P-S sequences} If $c < \frac{\alpha-s}{2(n-s)} \mu_{\gamma,s, \alpha}(\R^n)^{\frac{n-s}{\alpha-s}} $, then every Palais-Smale sequence $(u_k)_{k\in \N}$ for $\Phi$ at level $c$ has a convergent subsequence in $\homega$.
\end{proposition}

\begin{proof}[Proof of Proposition \ref{Prop M-P: existence of strong limit for all P-S sequences}]
Assume $c < \frac{\alpha-s}{2(n-s)}  \mur^{\frac{n-s}{\alpha-s}}$ and  let$(u_k)_{k\in \N} \in \homega$ be a Palais-Smale sequence for $\Phi$ at level $c,$ that is $\Phi(u_k) \to c$ and 
\begin{equation}\label{Palais-Smale conditions: perturbed nonlinearity}
 \ \Phi'(u_k) \to 0 \quad \text{ in } (\homega)',
\end{equation}
where   $(\homega)'$ denotes the dual of $ \homega$.\\
We first prove that $(u_k)_{k\in \N}$ is bounded in $\homega.$ One can  use $u_k \in \homega$ as a test function in (\ref{Palais-Smale conditions: perturbed nonlinearity})  to get that 
\begin{equation}
\|| u_k\||^2 -  \lambda \io u_k^2 dx  =   \int_{\Omega} \frac{(u_k)_+^{2_{\alpha}^*(s)}}{|x|^{s}} dx + \int_{\Omega} h (u_k)_+^q dx +o(\| |u_k\| |) \quad \text{ as } k \to \infty.
\end{equation}
On the other hand, from the definition of $\Phi,$ we deduce that 
\begin{equation}
\|| u_k\||^2 -  \lambda \io u_k^2 dx   = 2 \Phi(u_k) +\frac{2}{2_{\alpha}^*(s)}\int_{\Omega} \frac{(u_k)_+^{2_{\alpha}^*(s)}}{|x|^{s}} + \frac{2}{q} \int_{\Omega} h (u_k)_+^q dx.
\end{equation}
It follows from the last two identities that as $ k \to \infty$,
\begin{equation}
2 \Phi_q(u_k) = \left(1- \frac{2}{2_{\alpha}^*(s)}\right) \int_{\Omega} \frac{(u_k)_+^{2_{\alpha}^*(s)}}{|x|^{s}} dx +\left(1 - \frac{2}{q} \right) \int_{\Omega} h (u_k)_+^q dx + o(\|| u_k \||).
\end{equation}
This coupled with the Palais-Smale condition $\Phi(u_k) \to c,$ and the fact that $h\ge 0$ yield 
\begin{align*}
2 c &= \left(1- \frac{2}{2_{\alpha}^*(s)}\right) \int_{\Omega} \frac{(u_k)_+^{2_{\alpha}^*(s)}}{|x|^{s}} dx  +\left(1 - \frac{2}{q} \right) \int_{\Omega} h (u_k)_+^q dx + o(1)\\
& \ge  \left(1- \frac{2}{2_{\alpha}^*(s)}\right) \int_{\Omega} \frac{(u_k)_+^{2_{\alpha}^*(s)}}{|x|^{s}} dx +o
(1) \quad \text{ as } k \to \infty.
\end{align*}
Thus, 
$$\left(1 - \frac{2}{q} \right) \int_{\Omega} h (u_k)_+^q dx = O(1)  \quad \text{ as } k \to \infty.$$
We finally obtain 
$$ \|| u_k \||^2 - \lambda \io u_k^2 dx  \le  O(1) + o(\|| u_k\||) \quad \text{ as } k \to \infty. $$
Using that $\lambda < \lambda_1(\LO)$ and $ \gamma < \gamma_H(\alpha),$
we get that
\begin{align*}
0 & < \left(1 - \frac{\gamma}{\gamma_H(\alpha)} \right) \left(1 - \frac{\lambda}{\lambda_1(\LO)} \right) \| u_k\|_{\homega}^2\\
& \le \left(1 - \frac{\lambda}{\lambda_1(\LO)} \right) \|| u_k\||^2\\
&  \le  \|| u_k\||^2 - \lambda \io u_k^2 dx  \le  O(1) + o(\|| u_k \||) \quad \text{ as } k \to \infty. 
\end{align*}
We then deduce that 
$(u_k)_{k\in \N}$ is bounded in $\homega,$  which implies that there exists $u \in \homega$ such that, up to a subsequence, 
\begin{align}\label{Formula M-P: embedding results}
\bs
& (1) \quad  u_k \rightharpoonup  u \text{ weakly in }  \homega. \\
& (2)  \quad u_k  \to u \text{ strongly in }  L^{p_1}(\Omega) \ \text{ for all } p_1 \in [2, 2^*_\alpha).\\
& (3) \quad  u_k  \to u \text{ strongly in }  L^{p_2}(\Omega, |x|^{-s}dx) \ \text{ for all } p_2 \in [2, 2^*_\alpha(s)).
\es
\end{align}
We now claim that, up to a subsequence, we have
\begin{equation}\label{Formula M-P: (u_k-u) relation of the norm and H-S term}
\|| u_k - u\||^2   =   \int_{\Omega} \frac{(u_k-u)_+^{2_{\alpha}^*(s)}}{|x|^{s}} dx +o(1) \quad \text{ as } k \to \infty,
\end{equation}
and 
\begin{equation}\label{Formula M-P: d . norm(u_k-u) < c +o(1) }
\frac{\alpha-s}{2(n-s)} \|| u_k - u\||^2  \le c  +o(1) \quad \text{ as } k \to \infty.
\end{equation}
Indeed, straightforward computations yield
\begin{align}\label{Formula M-P: testing u_k-u at Phi-prim}
\bs
o(1) &= \langle \Phi'(u_k) - \Phi'(u), u_k - u \rangle \\
& =  \|| u_k - u\||^2 - \lambda \io (u_k - u)^2   \\
& - \int_{\Omega}  (u_k - u)  \frac{\left(( u_k )_+^{2_{\alpha}^*(s)-1} -  u_+^{2_{\alpha}^*(s)-1}\right)}{|x|^{s}} dx\\
& + \int_{\Omega} h (u_k - u)  \left[ ( u_k )_+^{q-1} - u_+^{q-1} \right] dx   \quad \text{ as } k \to \infty. 
\es
\end{align}
We first write

\begin{align*}
\io   (u_k - u)  \frac{\left(( u_k )_+^{2_{\alpha}^*(s)-1} -  u_+^{2_{\alpha}^*(s)-1}\right)}{|x|^{s}} dx &= \io     \frac{( u_k )_+^{2_{\alpha}^*(s)} }{|x|^{s}} dx - \io    u _k \frac{  u_+^{2_{\alpha}^*(s)-1} }{|x|^{s}} dx\\
& - \io    u  \frac{  (u_k)_+^{2_{\alpha}^*(s)-1} }{|x|^{s}} dx + \io     \frac{ u_+^{2_{\alpha}^*(s)} }{|x|^{s}} dx.     
\end{align*}
It now follows from integral theory that 
$$ \lim\limits_{k \to \infty} \io    u _k \frac{  u_+^{2_{\alpha}^*(s)-1} }{|x|^{s}} dx  = \io     \frac{ u_+^{2_{\alpha}^*(s)} }{|x|^{s}} dx = \lim\limits_{k \to \infty}  \io    u  \frac{  (u_k)_+^{2_{\alpha}^*(s)-1} }{|x|^{s}} dx. $$
So, we get that 
\begin{align*}
\io   (u_k - u)  \frac{\left(( u_k )_+^{2_{\alpha}^*(s)-1} -  u_+^{2_{\alpha}^*(s)-1}\right)}{|x|^{s}} dx &= \io     \frac{( u_k )_+^{2_{\alpha}^*(s)} }{|x|^{s}} dx - \io     \frac{ u_+^{2_{\alpha}^*(s)} }{|x|^{s}} dx.     
\end{align*}
In order to deal with the right hand side of the last identity, we use the following basic inequality:
$$ \left|  ( u_k )_+^{2_{\alpha}^*(s)}  -   u_+^{2_{\alpha}^*(s)}    -      (u_k - u)_+^{2_{\alpha}^*(s)}  \right|  \le c \Bigg(  u_+^{2_{\alpha}^*(s)-1}  |u_k -u| +  (u_k -u)_+^{2_{\alpha}^*(s)-1}  |u| \Bigg),$$ for some constant $c>0.$
We multiply both sides of the above inequality by $|x|^{-s}$ and take integral over $\Omega,$ and then use (\ref{Formula M-P: embedding results}) to get that

\begin{align*}
\lim\limits_{k \to \infty} \io  \frac{( u_k )_+^{2_{\alpha}^*(s)} - (u_k - u)_+^{2_{\alpha}^*(s)} }{|x|^s}  dx  =  \io  \frac{u _+^{2_{\alpha}^*(s)}}{|x|^s}   dx.
\end{align*}  
We therefore have 
$$ \int_{\Omega}  (u_k - u)  \frac{\left(( u_k )_+^{2_{\alpha}^*(s)-1} -  u_+^{2_{\alpha}^*(s)-1}\right)}{|x|^{s}} dx = \int_{\Omega} \frac{(u_k-u)_+^{2_{\alpha}^*(s)}}{|x|^{s}} dx +o(1)  \text{ as } k \to \infty. $$
In addition, the embeddings (\ref{Formula M-P: embedding results}) yield that
$$ \io (u_k - u)^2 = o(1) \quad \text{ as } k \to \infty,   $$ and
$$ \int_{\Omega} h (u_k - u)  \left[ ( u_k )_+^{q-1} - u_+^{q-1} \right] dx = \io (u_k - u)_+^q dx + o(1) = o(1)  \ \text{ as } k \to \infty.$$
Plugging back the last three estimates into (\ref{Formula M-P: testing u_k-u at Phi-prim}) gives (\ref{Formula M-P: (u_k-u) relation of the norm and H-S term}). On the other hand, since $u$ is a weak solution of (\ref{Main problem: M-P solutions}), then $\Phi(u)\ge 0,$ and since $\Phi(u_k) \to c $ as $k \to \infty,$ it follows that  $ \frac{\alpha-s}{2(n-s)} \| u_k - u\|^2  \le c  +o(1).$ This proves the claim. \\
We now show that 
\begin{equation}\label{Formula M-P: u_k strongly in homega}
\lim\limits_{k \to \infty} u_k = u  \ \text{  in } \homega.
\end{equation}
Indeed, 
test 
the inequality (\ref{Problem: the best fractional H-Sconstant on bounded domain }) on $u_k-u,$ and use (\ref{Formula M-P: embedding results}) and (\ref{Formula: mu(Omega)= mu(R^n)}) to obtain that 
\begin{equation}\label{Formula M-P: testing (u_k-u) at H-S inequalities}
\int_{\Omega} \frac{(u_k-u)_+^{2_{\alpha}^*(s)}}{|x|^{s}} dx \le \mur^{-\frac{\crits}{2}} \|| u_k - u \||^{\crits} +o(1).
\end{equation}
Combining this with (\ref{Formula M-P: (u_k-u) relation of the norm and H-S term}), we get 

$$\|| u_k-u \||^2 \left(   1 -  \mur^{-\frac{\crits}{2}}  \|| u_k-u \||^{\crits - 2}     +o(1) \right) \le o(1).$$
It then follows from the last inequality and (\ref{Formula M-P: d . norm(u_k-u) < c +o(1) }) that 

$$\left(   1 -  \mur^{-\frac{\crits}{2}} \Big(\frac{2(n-s)}{\alpha-s} c \Big)^{\frac{\crits - 2}{2}}     +o(1) \right) \|| u_k-u \||^2  \le o(1).$$
Note  that the assumption $ c < \frac{\alpha-s}{2(n-s)}  \mur^{\frac{n-s}{\alpha-s}}$ implies that $$\Big( \frac{2(n-s)}{\alpha-s} c \Big)^{\frac{\crits - 2}{2}} < \mur^{\frac{\crits}{2}},$$ and therefore
$$\left(   1 -  \mur^{-\frac{\crits}{2}} \Big(\frac{2(n-s)}{\alpha-s} c \Big)^{\frac{\crits - 2}{2}}    \right) > 0.$$
Thus,  $\|| u_k-u \|| \to 0$ as $k \to \infty,$ and this proves (\ref{Formula M-P: u_k strongly in homega}).\\
Finally, 
we have that $\Phi(u) = c,$ since the functional is continuous on $\homega$.

\end{proof}

\subsubsection{Mountain pass geometry and existence of a Palais-Smale sequence}

In this subsection, we prove the following
\begin{proposition}\label{Prop M-P: existence of P-S sequence and level set} For every $w\in \homega \setminus\{0\}$ with $w\geq 0$, there exists an energy level $c$, with 
\begin{equation}
0<c \leq 
\sup\limits_{t \ge 0} \Phi(t w),
\end{equation}
 and  a Palais-Smale sequence $(u_k)_k$ for $\Phi$ at level $c,$ that is 
\begin{equation*}
\Phi(u_k) \to c \ \text{ and }  \ \Phi'(u_k) \to 0 \text{ in } (\homega)'.
\end{equation*}

\end{proposition} 


\begin{proof}
We show  that the functional $\Phi$ satisfies the  hypotheses of the mountain pass lemma \ref{Theorem MPT- Ambrosetti-Rabinowitz version}. It is standard to show that $\Phi \in C^1(\homega)$ and clearly  $\Phi(0)=0$, so that (a) of Lemma \ref{Theorem MPT- Ambrosetti-Rabinowitz version} is satisfied. 

For (b), we show that $0$ is a strict local minimum. For that, we first need to recall  the definition of $\lambda_1(\LO),$ which is  the first eigenvalue of the operator $\LO$ with Dirichlet boundary condition, and the best constant in the fractional Hardy-Sobolev inequality on domain $\Omega,$ that is 
\begin{equation} \label{Problem: the best fractional H-Sconstant on bounded domain }
\mu_{\gamma,s,\alpha}(\Omega):= \inf\limits_{u \in \homega\setminus \{0\}} \frac{ \frac{c_{n,\alpha}}{2}\int_{\R^n}  \int_{\R^n} \frac{|u(x)- u(y)|^2}{|x-y|^{n+\alpha}} dxdy- \gamma \int_{\Omega} \frac{|u|^2}{|x|^{\alpha}}dx }{(\int_{\Omega} \frac{|u|^{2_{\alpha}^*(s)}}{|x|^{s}}dx)^\frac{2}{2_{\alpha}^*(s)}}.
\end{equation}
As in the local case, one can show by translating, scaling and cutting off the extremals of $ \mu_{\gamma,s,\alpha}(\R^n)$ that 
\begin{equation}\label{Formula: mu(Omega)= mu(R^n)}
\mu_{\gamma,s,\alpha}(\Omega)=\mu_{\gamma,s,\alpha}(\R^n).
\end{equation}

See Proposition 6.1 in \cite{GFSZ} for more detail. Therefore, we  have that 
$$\lambda_1(\LO)  \int_{\Omega} |w|^2 dx \le   \||w\||^2 \  \text{ and } \  \mur (\int_{\Omega} \frac{|w|^{2_{\alpha}^*(s)}}{|x|^{s}}dx)^\frac{2}{2_{\alpha}^*(s)}\le  \|| w\||^2.$$

In addition, it follows from $(\ref{Formula M-P: embedding results})_2$ that there exists a positive constant $S>0$ such that 
$$ S  (\int_{\Omega} h |w|^q dx)^{\frac{2}{q}} \le   \|| w\||^2. $$
Hence,
\begin{equation*}
\begin{aligned}
\Phi(w) &\ge  \frac{1}{2} \|| w\||^2 -\frac{1}{2} \frac{\lambda}{\lambda_1(\LO)}\|| w\||^2 -  \frac{1}{q} S^{-\frac{q}{2}} \|| w\||^{q}  \\ 
& \hspace{5.2cm}   -\frac{1}{2_{\alpha}^*(s)} \mur^{-\frac{2_{\alpha}^*(s)}{2}} \|| w\||^{2_{\alpha}^*(s)} \\
& \ = \||w\||^2 \Bigg( \frac{1}{2}(1 - \frac{\lambda}{\lambda_1(\LO)} )  -\frac{1}{q} S^{-\frac{q}{2}} \|| w\||^{q-2}\\
& \hspace{5.2cm} -\frac{1}{2_{\alpha}^*(s)} \mur^{-\frac{2_{\alpha}^*(s)}{2}} \|| w\||^{2_{\alpha}^*(s)-2} \Big).
\end{aligned}
\end{equation*}
Since $\lambda <\lambda_1(\LO),$ $ q \in (2, 2^*_\alpha)$ and $s \in [0,\alpha),$ we have that $1 - \frac{\lambda}{\lambda_1(\LO)} >0,$  $q-2 >0 $ and $2_{\alpha}^*(s)-2 >0,$ respectively. Thus,  we can find $R>0$ such that  $\Phi(w) \ge \rho$ for all $w \in \homega$ with $\|w\|_{\homega} = R.$\\
  Regarding (c), we have 
$$\Phi(tw) = \frac{t^2}{2} \|| w\||^2 -\frac{t^2 \lambda}{2} \io w^2 dx -\frac{ t^{ 2_{\alpha}^*(s)}  }{2_{\alpha}^*(s)} \io \frac{w_+^{2_{\alpha}^*(s)}}{|x|^{s}}  dx- \frac{t^q}{q} \io h w_+^q dx  ,$$
hence $ \lim\limits_{t \to \infty} \Phi(tw)=  -\infty$ for any $w \in \homega \setminus \{0\}$ with $w_+ \not \equiv 0$, which means that    there exists $t_w>0$ such that $\|t_w w\|_{\homega} > R$ and $\Phi(tw) <0,$ for $t\ge t_w.$ In other words,
$$
0 <\rho \leq \inf \{\Phi (w); \|w\|_{\homega} = R\} \leq c = \inf\limits_{\gamma \in {\mathcal F}}\sup\limits_{t\in [0, 1]} \Phi (\gamma (t))\leq \sup\limits_{t \ge 0} \Phi(t w),
$$
where ${\mathcal F}$ is the class of all path $\gamma \in C([0, 1]; \homega)$ with $\gamma (0)=0$ and $\gamma (1)=t_ww$. \\
The rest follows from the Ambrosetti-Rabinowitz lemma  \ref{Theorem MPT- Ambrosetti-Rabinowitz version}. 

\end{proof}

\subsubsection{End of Proof of Theorem \ref{Thm M-P: main condition of existence}}
We assume that there exists there exists $w \in \homega,$ $w \not\equiv 0,$ $w \ge 0$ such that $$\sup\limits_{t \ge 0} \Phi(t w) < \frac{\alpha-s}{2(n-s)}  \mur^{\frac{n-s}{\alpha-s}}.$$
\textit{Existence:}
 It follows from Proposition \ref{Prop M-P: existence of P-S sequence and level set} that there exist  $c_w>0$ and a Palais-Smale sequence $(u_k)_{k \in \N}$ for $\Phi$ at level $c_w$ such that $$c_w \le \sup\limits_{t \ge 0} \Phi(t w).$$  It is obvious now that  $c_w < \frac{\alpha-s}{2(n-s)}  \mur^{\frac{n-s}{\alpha-s}}.$ So,  we can apply Proposition \ref{Prop M-P: existence of strong limit for all P-S sequences} to  get that there exists $u \in \homega$ such that, up to a subsequence, $ u_k \to u \text{ strongly in } \homega$ as $k \to \infty,$ and $ \Phi'(u) = 0.$ This means that $u$ is a solution for (\ref{Main problem: M-P solutions}).\\
\textit{Positivity:} We test  (\ref{Main problem: M-P solutions}) against $u_-\in\homega.$ Then,  arguing as in the proof of Lemma 5.2 in  \cite{GFSZ}, we get that $\|u_-\|_{\h} = 0,$ which implies that  $u_- \equiv 0$, and therefore $u \ge 0 $ on $\Omega.$

\subsection{Test functions estimates for the non-local case}\label{Sec:Test Functions}

This section is devoted to prove an important result (Proposition \ref{Prop M-P: estimate at the test functions}) which plays a crucial role in the proof of Theorem \ref{Thm M-P: Main result }.  Since $\Phi$ satisfies the Palais-Smale condition only up to level $ \frac{\alpha-s}{2(n-s)}  \mur^{\frac{n-s}{\alpha-s}}$, we need to check which conditions on $\gamma$, $q$, $h$ and the mass of the domain $\Omega$ guarantee that there exists a $w\in \homega$ such that 
\begin{equation}\label{Condition:General condition}
\sup\limits_{t \ge 0} \Phi(t w) < \frac{\alpha-s}{2(n-s)}  \mur^{\frac{n-s}{\alpha-s}}.
\end{equation}

We shall use the test functions constructed in \cite{GFSZ} to obtain the general condition of existence (\ref{Condition:General condition}).  More precisely, we use the extremal of $\mur$ and its profile at zero and infinity  to introduce appropriate test functions. We need such test functions to estimate the general existence condition (\ref{Condition:General condition}). Unlike the case of the Laplacian ($\alpha =2$), no explicit formula is known for the best constant $\mu_{\gamma,s,\alpha}(\R^n)$ nor for the extremals where it is achieved.
 The existence of such extremals  was proved in \cite{GS} under certain condition on $\gamma$ and $s.$  We therefore need to describe their asymptotic profile whenever they exist. This was recently considered in Ghoussoub-Robert-Shakerian-Zhao \cite{GFSZ}, where the following is proved:



\begin{theorem}\label{th:asymp:ext}
Assume $0 \le s<\alpha<2,$ $n> \alpha$ and $\displaystyle 0\leq\gamma<\gamma_H(\alpha)$. Then, any positive  extremal $\displaystyle u\in \h$ for $\mu_{\gamma,s,\alpha}(\R^n)$ satisfies $u\in C^1(\rn\setminus\{0\})$ and  
\begin{equation}
\lim_{x\to 0}|x|^{\Bm}u(x)=\lambda_0\hbox{ and }\lim_{|x|\to \infty}|x|^{\Bp}u(x)=\lambda_\infty,
\end{equation}
where $\lambda_0,\lambda_\infty>0$ and $\Bm$ (resp., $\Bp$) is the unique solution in 
$\left(0,\frac{n-\alpha}{2}\right)$ (resp., in $\left(\frac{n-\alpha}{2},n-\alpha\right))$ of the equation 

$$\Psi_{n,\alpha}(t):= 2^\alpha \frac{\Gamma(\frac{n-t}{2})\Gamma(\frac{\alpha+t}{2})}{  \Gamma(\frac{n-t-\alpha}{2})\Gamma(\frac{t}{2} ) }= \gamma.$$

\end{theorem}

\begin{remark}\label{Remark:fundamental solution of LO on R^n}
We point out that the functions $u_1(x) = |x|^{- \Bm}$ and  $u_2(x) = |x|^{- \Bp}$ are
the fundamental solutions for the fractional Hardy-Schr\"odinger operator $ \LO:= (-\Delta)^{\frac{\alpha}{2}}-\frac{\gamma}{|x|^\alpha} $ on  $\R^n.$ Indeed, a straightforward computation  yields  (see Section 2 in \cite{GFSZ})
$$ \LO |x|^{-\beta} =   (\Psi_{n,\alpha}(\beta) - \gamma ) |x|^{-\beta} = 0 \text{ on } \R^n \quad  \hbox{ for  } \beta \in \{\Bm, \Bp\},$$
which implies that $\Bp$ and $\Bm$ satisfy $\Psi_{n,\alpha}(\beta) = \gamma.$

\end{remark}

We then stablish the following.

\begin{proposition}\label{Prop M-P: estimate at the test functions}
There exists $\tau_l >0 $ for $l = 1,..,5$ such that\\

1) If $0 \le \gamma \le \gamma_{crit}(\alpha),$ then

\begin{equation}\label{Formula M-P: Phi(u eps): NoN-critical}
\sup\limits_{t \ge 0}\Phi(t v_\eps) = \frac{\alpha-s}{2(n-s)}  \mur^{\frac{n-s}{\alpha-s}}  - \tau_1  h(0) \eps^{n - q\frac{n - \alpha}{2}} + o(\eps^{n- q\frac{n - \alpha}{2}}); 
\end{equation}

2) If $ \gamma_{crit}(\alpha) <  \gamma < \gamma_H(\alpha),$ then
\begin{align}\label{Formula M-P: Phi(u eps): Critical}
\bs
\sup\limits_{t \ge 0}\Phi(t v_\eps) &= \frac{\alpha-s}{2(n-s)}  \mur^{\frac{n-s}{\alpha-s}} \\
&+ \left\{\begin{array}{rl}
\displaystyle   - \tau_2  h(0) \eps^{n - q\frac{n - \alpha}{2}} + o(\eps^{n- q\frac{n - \alpha}{2}}) \,\,\,\,\,\,\,\,\,\,\,\,\,\,\,\,\,\,\,\,\,\,\,\,\,\,\,\,\,\,\,\,\,\,\,\,\,\,\,\,\,\,\,\,\,\,\,\,\,\,\,\,\,\,\,\,\,\,\,\,\,\,\,\,\,\,\,\, & \text{if }  q > q_{crit} \vspace{0.1cm}\\
- (\tau_3  h(0) + \tau_4 m^\alpha_{\gamma,\lambda}) \eps^{\Bp-\Bm} + o(\eps^{\Bp-\Bm})  & \text{if } q= q_{crit} \\ - \tau_5  m^\alpha_{\gamma,\lambda} \eps^{\Bp-\Bm} + o(\eps^{\Bp-\Bm})\,\,\,\,\,\,\,\,\,\,\,\,\,\,\,\,\,\,\,\,\,\,\,\,\,\,\,\,\, & \text{if } q <  q_{crit},
\end{array}\right. 
\es
\end{align}
where $q_{crit} := 2^*_\alpha - 2 \frac{\bp-\bm}{n-\alpha} \in  (2,2^*_\alpha).$

\end{proposition}

\begin{remark}
Throughout this section, we may use the following notations:

$$\bp:= \Bp \quad \text{ and  } \quad \bm:= \Bm$$

\end{remark}

\subsubsection{Test function for non-critical case:}\label{susec:test:non-critical}

\medskip\noindent We fix  cut-off function
$\eta \in C_c^{\infty} (\Omega) $ such that
\begin{equation}\label{def:eta}
\eta\equiv 1\hbox{ in }B_\delta(0)\hbox{ and }\eta\equiv 0\hbox{ in }\rn\setminus B_{2\delta}(0)\hbox{ with }B_{4\delta}(0)\subset \Omega.
\end{equation}
Let $U_\alpha \in \h$ be an extremal for $\mu_{\gamma,s,\alpha}(\rn)$. It follows from Theorem \ref{th:asymp:ext} that, up to multipliying by a nonzero constant, $U_\alpha$ satisfies for some $\kappa>0$, 
\begin{equation}\label{eq:U}
(-\Delta)^{\frac{\alpha}{2}}U_\alpha-\frac{\gamma}{|x|^\alpha}U_\alpha=\kappa \frac{U_\alpha^{\crits-1}}{|x|^s}\hbox{ weakly in }\h.
\end{equation}
 Moreover, $U_\alpha \in C^1(\rn\setminus\{0\})$, $U_\alpha>0$ and 
\begin{equation}\label{asymp:U}
\lim_{|x|\to \infty}|x|^{\bp} U_\alpha (x)=1.
\end{equation}

Consider
$$u_\epsilon(x):=\eps^{-\frac{n-\alpha}{2}}U(\eps^{-1}x)\hbox{ for }x\in\rn\setminus\{0\}.$$
It follows from Proposition 3.1 in \cite{GFSZ} that $U_{\eps,\alpha}:= \eta u_{\eps,\alpha}\in \homega.$

\subsubsection{Test function for the critical case:}\label{subsec:test:critical}

The test function in the critical case is more delicate and has more complicated constructure compare to the non-critical case. In order to build a suitable test function, 
we use the fractional Hardy singular interior mass of a domain associted to the operator $\LO$ which was introduced by Ghoussoub-Robert-Shakerian-Zhao \cite{GFSZ}.

We  first summarize the results which will be needed in this section: Assume that $\gamma>\gamma_{crit}(\alpha).$ It follows from Theorem 1.2 in \cite{GFSZ} 
 that there exists $H: \Omega\setminus\{0\}\to\rr$ such that 
\begin{equation*}
\left\{\begin{array}{ll}
H\in C^1(\Omega\setminus\{0\})\;,\;\xi H\in \homega &\hbox{ for all }\xi\in C_c^\infty(\rn\setminus\{0\}),\\
(-\Delta)^{\frac{\alpha}{2}}H-\left(\frac{\gamma}{|x|^\alpha}+a)\right)H=0& \hbox{ weakly in }\Omega\setminus\{0\}\\
H>0&\hbox{ in }\Omega\setminus\{0\}\\
H=0&\hbox{ in }\partial\Omega\\
{\rm and} \,\, \lim_{x\to 0}|x|^{\bp}H(x)=1.&
\end{array}\right.
\end{equation*}
Note that the second identity means that for any $\varphi\in C^\infty_c(\Omega\setminus\{0\})$, we have that
\begin{equation}\label{eq:H}
\frac{C_{n,\alpha}}{2}\int_{(\rn)^2}\frac{(H(x)-H(y))(\varphi(x)-\varphi(y))}{|x-y|^{n+\alpha}}\, dxdy-\int_{\rn}\left(\frac{\gamma}{|x|^\alpha}+a\right) H\varphi\, dx=0.
\end{equation}
 Let now $\eta$ be as in \eqref{def:eta}. Following the construction of the singular function $H$ in the proof of Theorem 1.2 in \cite{GFSZ}, we get that there exists $g\in \homega$ such that 
$$H(x) := \frac{\eta(x)}{|x|^{\bp}} + g(x)\hbox{ for }x\in \Omega\setminus\{0\},$$
where $g$ satisfies  
\begin{equation}\label{eq:t}
(-\Delta)^{\frac{\alpha}{2}}g-\left(\frac{\gamma}{|x|^\alpha}+a\right)g=f,
\end{equation}
\begin{equation}\label{def:mass}
g (x)= \frac{m^\alpha_{\gamma,a}(\Omega)}{|x|^{\bm}}+o\left(\frac{1}{|x|^{\bm}}\right)\hbox{ as }x\to 0,\hbox{ and } |g(x)|\leq C|x|^{-\bm}\hbox{ for all }x\in\Omega.
\end{equation}

Here $m^\alpha_{\gamma,a}(\Omega) $ is the fractional Hardy singular interior mass of a domain associted to the operator $\LO.$  We refer the readers to Section 5 of \cite{GFSZ} for the definition and properties of the Mass in detail.

Define the test function as 
$$ \Te(x) = \eta u_\epsilon(x) + \epsilon^{\frac{\bp-\bm}{2}} g(x) \quad\text{ for all }  x \in \overline{\Omega} \setminus \{0\},$$
where 
$$u_{\epsilon,\alpha}(x):=\eps^{-\frac{n-\alpha}{2}}U_\alpha(\eps^{-1}x)\hbox{ for }x\in\rn\setminus\{0\},$$
and $U\in \h$ is such that $U>0$, $U\in C^1(\rn\setminus\{0\})$ and satisfies \eqref{eq:U} above for some $\kappa>0$ and also \eqref{asymp:U}. It is easy to see that $ \Te \in \homega$ for all  $\epsilon >0$.


\begin{remark}\label{Remark:properties of Bp and Bm}
One can summarize the properties of $\Bp$ and $\Bm$ which will be used freely in this section:

\begin{itemize}
\item  $\bm+\bp=n-\alpha$

\item $\bp-\bm>\alpha$ when $\gamma<\gamma_{crit}(\alpha)$
 \item  $\bp-\bm=\alpha$ when $\gamma=\gamma_{crit}(\alpha).$

\item $\bp-\bm < \alpha$ when $\gamma > \gamma_{crit}(\alpha)$ 
\end{itemize}

\end{remark}

\subsubsection{Proof of Proposition \ref{Prop M-P: estimate at the test functions}}\label{Subsection:estimates for non-local case}
We shall use the test functions $U_\eps$(resp., $T_\eps$) constructed in Subsection \ref{susec:test:non-critical}(resp., Subsection \ref{subsec:test:critical}) for  the case when the operator $\LO$ is non-critical (resp., critical) to obtain the general condition of existence. We define the test-functions then as follows:

\begin{equation}\label{Formula M-P: test functions}
v_{\eps,\alpha} =\left\{\begin{array}{rl}
\displaystyle U_{\eps,\alpha}= \eta u_{\eps,\alpha} \,\,\,\,\,\,\,\,\,\,\,\,\,\,\,\,\,\,\,\,\,\,\,\,\,\,\,\,\,\,\,\,\,\,\,\,\,\,\,\,\,\,\,\, & \text{if }  0 \le \gamma \le \gamma_{crit}(\alpha) \vspace{0.1cm}\\
\Te := U_\eps + \epsilon^{\frac{\bp-\bm}{2}} g(x) & \text{if } \gamma_{crit}(\alpha) < \gamma < \gamma_H(\alpha).
\end{array}\right.
\end{equation}

We expand $\Phi(t v_\epsilon)$ in the following way:
$$\Phi(t v_\epsilon) = \frac{t^2}{2} I_\eps - \frac{t^{\crits}}{\crits}J_\eps - \frac{t^q}{q} K_\eps \quad \text{ as } \epsilon \to 0,$$
where 
$$ I_\eps := \|| v_\eps \||^2 - \lambda \|v_\eps\|_{L^2(\Omega)}^2, \ J_\eps:= \io \frac{|v_\eps|^{\crits}}{|x|^s} dx  \text{ and } K_\eps:= \io h |v_\eps|^q dx.$$
We start by recalling the important estimates obtained in \cite{GFSZ}. Indeed, It was proved there that there exist positive constant $c_1,c_2,c_3$ such that 
\begin{small}
\begin{equation}\label{Estimate:I_epsilon}
I_\eps =\left\{\begin{array}{rl}
\displaystyle \kappa \ir \frac{U^{\crits}}{|x|^s} dx- c_1 \lambda \eps^{\alpha}  + o(\eps^{\alpha}) \,\,\,\,\,\,\,\,\,\,\,\,\,\,\,\, \,\,\,\,\,\,\,\,\,\,\,\,\,\,\,\,\,\,\,\,\,\,\,\,\,\,\,\,\,\,\,\,\,\,\,\,\,\,\,\,\,\,\,\,\,\,\,\,\,\,\,\,\,\,\,\,\,\,\,\,\,\,\,\,\,\,\,\,\,\,\,\,\,\,\,\,\,\,\,\,\,\,\,\,\,\,\,\,\,\,\,\,\,\,\,\,\,\,\,& \text{if }  0 \le \gamma < \gamma_{crit}(\alpha) \vspace{0.1cm}\\
\kappa \displaystyle \ir \frac{U^{\crits}}{|x|^s} dx - c_2 \lambda \eps^{\alpha} \ln \eps^{-1} + o(\eps^{\alpha} \ln \eps^{-1})\,\,\,\,\,\,\,\,\,\,\,\,\,\,\,\,\,\,\,\,\,\,\,\,\,\,\,\,\,\,\,\,\,\,\,\,\,\,\,\,\,\,\,\,\,\,\,\,\,\,\,\,\,\,\,\,\,\,\,\,\,\,\,\,\,\,\,\,\,\,\,\,\,\,\,\,\,\,\,\,\,\,\,\,\,\, & \text{if } \gamma = \gamma_{crit}(\alpha) \\
\kappa \displaystyle \ir \frac{U^{\crits}}{|x|^s} dx - c_3 m^\alpha_{\gamma,\lambda} \eps^{\bp-\bm}+2 \kappa \eps^{\frac{\bp-\bm}{2}} \theta_\eps  + o( \eps^{\bp-\bm})   & \text{if } \gamma_{crit}(\alpha) < \gamma < \gamma_H(\alpha),
\end{array}\right.
\end{equation}
\end{small}
as $\eps \to 0,$ and 
\begin{equation}\label{Estimate:J_epsilon}
J_\eps =\left\{\begin{array}{rl}
\displaystyle   \ir \frac{U^{\crits}}{|x|^s} dx + o( \eps^{\bp-\bm}) \,\,\,\,\,\,\,\,\,\,\,\,\,\,\,\,\,\,\,\,\,\,\,\,\,\,\,\,\,\,\,\,\,\,\,\,\,\,\,\,\,\,\,\,\,\,\,\,\,\,\,\,\,\,\, & \text{if }  0 \le \gamma \le \gamma_{crit}(\alpha) \vspace{0.1cm}\\
\displaystyle \ir \frac{U^{\crits}}{|x|^s} dx+ \crits  \eps^{\frac{\bp-\bm}{2}} \theta_\eps + o( \eps^{\bp-\bm})   & \text{if } \gamma_{crit}(\alpha) < \gamma < \gamma_H(\alpha),
\end{array}\right.
\end{equation}
as $\eps \to 0.$  Here,  $\theta_\eps:=\int_{\rn}\frac{\ue^{\crits-1}\eta g}{|x|^s}\, dx$ and we have $ \lim\limits_{\epsilon \to 0}\theta_\eps =0;$ see (57) in  \cite{GFSZ}. We are then left with estimating $K_\eps.$\\

\smallskip\noindent \textbf{Estimate for $K_\eps:$} We will consider two following cases.\\

\textit{Case 1:} $0 \le \gamma \le \gamma_{crit}(\alpha).$  We split  $K_\eps$ into two integrals  as follows

\begin{equation*}\label{Formula M-P: Split K eps into two integrals}
  K_\eps = \io h |U_\eps|^q dx = \int_{B_\delta} h |U_\eps|^q dx + \int_{\Omega \setminus B_\delta} h |U_\eps|^q dx. 
\end{equation*}
We start by estimating the first term:  
\begin{align*}
\int_{B_\delta} h |U_\eps|^q dx &= \eps^{q\frac{\alpha-n}{2}} \int_{B_\delta} h(x) |U(\frac{x}{\eps})|^q dx =  \eps^{n - q\frac{n - \alpha}{2}} \int_{B_{\frac{\delta}{\eps}}} h(\eps X) |U(X)|^q dX\\
& =  \eps^{n - q\frac{n- \alpha}{2}} h(0) \int_{\R^n}  |U(X)|^q dX  + \eps^{n - q\frac{n - \alpha}{2}} \int_{\R^n \setminus B_{\frac{\delta}{\eps}}} h(\eps X) |U(X)|^q dX.
\end{align*}
Note that we used the change of variable $x = \eps X$ . From the asymptotic (\ref{asymp:U}) and the fact that $q > 2$, it then follows that  
\begin{align*}
\int_{B_\delta} h |U_\eps|^q dx &= \eps^{n - q\frac{n- \alpha}{2}} h(0) \int_{\R^n}  |U(X)|^q dX + O(\eps^{q\frac{\bp-\bm}{2}})\\
& = \eps^{n - q\frac{n- \alpha}{2}} h(0) \int_{\R^n}  |U(X)|^q dX + o(\eps^{\bp-\bm}). 
\end{align*}
Following the same argument that we treat the second integral in the last term yields that
$$\int_{\Omega \setminus B_\delta} h |U_\eps|^q dx = O(\eps^{q\frac{\bp-\bm}{2}})=o(\eps^{\bp-\bm}).$$
Therefore, 
\begin{align}\label{Formula M-P: Estimate for h U-eps^q : non-critical}
\int_{\Omega} h |U_\eps|^q dx &= \eps^{n - q\frac{n- \alpha}{2}} h(0) \left[ \int_{\R^n}  |U(X)|^q dX  \right]+ o(\eps^{\bp-\bm}). 
\end{align}
It now follows from Remark \ref{Remark:properties of Bp and Bm} that  
$$\bp-\bm  \ge \alpha \quad  \text{ if } 0\le\gamma\le\gamma_{crit}(\alpha).$$
On the other hand, the condition $2 < q < 2^*_\alpha$ implies that 
$$ 0 < n - q\frac{n - \alpha}{2} < \alpha.$$ 
Combining the last two inequalities, we then get that  $$ n - q \frac{n - \alpha}{2} < \bp-\bm,$$
and therefore 
\begin{align}\label{Formula M-P: Estimate for K_eps : non-critical}
K_\eps = \eps^{n - q\frac{n- \alpha}{2}} h(0) \left[ \int_{\R^n}  |U(X)|^q dX  \right]+ o(\eps^{n - q\frac{n - \alpha}{2}}), 
\end{align}
when $0 \le \gamma \le \gamma_{crit}(\alpha).$\\


\textit{Case 2:} $\gamma_{crit}(\alpha) < \gamma < \gamma_H(\alpha).$
In order to estimate $K_\eps$ in the critical case, we need the following inequality: For $q > 2,$ there exists $C=C(q)>0$ such that 
$$ ||X+Y|^q - |X|^q| - qXY|
X|^{q-2} | \le C (|X|^{q-2} Y^2 + |Y|^q) \quad \text{ for all } X,Y\in \R.  $$
We write
$$  K_\eps = \int_{\Omega} h |T_\eps|^q dx =\int_{B_\delta} h |U_\eps + \epsilon^{\frac{\bp-\bm}{2}} g(x) |^q dx + O(\eps^{q\frac{\bp-\bm}{2}}),$$
where the last term came from the fact that 
$$ \int_{\Omega \setminus B_\delta} h |T_\eps  |^q dx = O(\eps^{q\frac{\bp-\bm}{2}})=o(\eps^{\bp-\bm}).$$
Let now $X =  U_\eps$ and  $Y =  \epsilon^{\frac{\bp-\bm}{2}} g(x) $ in the above inequality. Taking integral from both sides then leads us to 
\begin{align*}
\int_{B_\delta} h |U_\eps + \epsilon^{\frac{\bp-\bm}{2}} g(x) |^q dx =  \int_{B_\delta} h |U_\eps|^q dx + q \epsilon^{\frac{\bp-\bm}{2}}  \int_{B_\delta} h g |U_\eps|^{q-1} dx + R_\eps,
\end{align*}
where 
\begin{align}\label{Formula M-P: Estimate for R_eps}
\bs
R_\eps &= O\left( \epsilon^{\bp-\bm} \int_{B_\delta} h  |U_\eps|^{q-2}  g^2 dx +   \epsilon^{q\frac{\bp-\bm}{2}} \int_{B_\delta}  h g^q dx\right)\\
&= O( \eps^{q\frac{\bp-\bm}{2}}) =  o(\epsilon^{\bp-\bm}).
\es
\end{align}
Regarding the second term, we have 
\begin{align}\label{Formula M-P: Estimate for term with power (q-1)}
\bs
q \epsilon^{\frac{\bp-\bm}{2}}  \int_{B_\delta} h g |U_\eps|^{q-1} dx& 
= O( \eps^{(\bp - \bm) + (n - q\frac{n - \alpha}{2})}) +O( \eps^{q\frac{\bp-\bm}{2}})\\
 & = o( \eps^{\bp-\bm})  \quad \qquad \text{ for all } q \in (2, 2^*_\alpha).
\es
\end{align}
Combining (\ref{Formula M-P: Estimate for h U-eps^q : non-critical}), (\ref{Formula M-P: Estimate for R_eps}) and (\ref{Formula M-P: Estimate for term with power (q-1)}), we get that there exist a constant $C>$ $0$ such that  
$$K_\eps = C h(0) \eps^{n - q\frac{n - \alpha}{2}} + o(\eps^{n- q\frac{n - \alpha}{2}}) + o(\epsilon^{\bp-\bm}) \quad \text{ if } \gamma_{crit}(\alpha) < \gamma < \gamma_H(\alpha).$$
We point out that the situation in the critical case is more delicate, and unlike the non-critical case, we have that both
$$\bp-\bm \text{ and }  n - q\frac{n - \alpha}{2} \text{ are in the interval } (0,\alpha).$$
Therefore, there is a competition between the terms  $ \eps^{\bp-\bm}$ and  $\eps^{n - q\frac{n - \alpha}{2}}.$ In order to find the threshold, namely $q_{crit},$ we equate the exponents of  the $\eps$ terms, and   solve the equation for $q$ to get that 
$$q_{crit} = 2^*_\alpha - 2 \frac{\bp-\bm}{n- \alpha} .$$ 
One should note that $q_{crit}  \in  (2,2^*_\alpha),$ since $\alpha > \bp-\bm >0$ in the critical case. This implies that 
\begin{align*}
\bs
\left\{\begin{array}{rl}
\displaystyle   o(\eps^{n- q\frac{n - \alpha}{2}}) + o(\eps^{\bp-\bm}) = o(\eps^{n- q\frac{n - \alpha}{2}})\,\,\,\,\,\,\,\,\,\,\,\,\,\,\,\,\,\,\,\,\,\,\,\,\,\,\,\,\,\,\,\,\,\,\,\,\,\,\,\,\,\,\,\,\,\,\,\,\, & \text{if }  q > q_{crit} \vspace{0.1cm}\\
o(\eps^{n- q\frac{n - \alpha}{2}}) + o(\eps^{\bp-\bm}) = o(\eps^{n- q\frac{n - \alpha}{2}})= o(\eps^{\bp-\bm})\, & \text{if } q= q_{crit} \\  o(\eps^{n- q\frac{n - \alpha}{2}}) + o(\eps^{\bp-\bm})= o(\eps^{\bp-\bm}) \,\,\,\,\,\,\,\,\,\,\,\,\,\,\,\,\,\,\,\,\,\,\,\,\,\,\,\,\,\,\,\,\,\,\,\,\,\,& \text{if } q <  q_{crit}.
\end{array}\right. 
\es
\end{align*}
We finally obtain 
\begin{align}\label{Formula M-P: Estimate for K_eps : citical }
\bs
K_\eps = \left\{\begin{array}{rl}
\displaystyle c_4 h(0) \eps^{n - q\frac{n - \alpha}{2}} + o(\eps^{n- q\frac{n - \alpha}{2}}) \,\,\,\,\,\,\,\,\,\,\,\,\,\,\,\,\,\,\,\,\,\,\,\,\,\,\,  & \text{if }  q > q_{crit} \vspace{0.1cm}\\
c_5 h(0) \eps^{\bp-\bm} + o(\eps^{\bp-\bm}) \,\,\,\,\,\, & \text{if } q= q_{crit}\\
o(\eps^{\bp-\bm}) \,\,\,\,\,\,\,\,\,\,\,\,\,\,\,\,\,\,\,\,\,\,\,\,\,\,\,\,\,\,\,\,\,\,\,\,\,\,\,\,\,\,\,\,\,\,\,\,\,\,\,\,\,\,\,\,\,\,\,\,\,\, & \text{if } q <  q_{crit}.
\end{array}\right. 
\es
\end{align}
for some $c_4,c_5>0,$ and as long as $\gamma_{crit}(\alpha) < \gamma < \gamma_H(\alpha).$\\

\smallskip\noindent We now define 
$$ I_0 := \lim\limits_{\eps \to 0} I_\eps  = \kappa \ir \frac{U^{\crits}}{|x|^s} dx \ \text{ and } \ J_0 := \lim\limits_{\eps \to 0} J_\eps  = \ir \frac{U^{\crits}}{|x|^s} dx, $$ and it is easy to check that  $$\lim\limits_{\eps \to 0} K_\eps = 0  \quad \text{ for all cases}.$$ 

\smallskip\noindent In the next step, we claim that, up to a subsequence of $(u_\eps)_{\eps >0 },$ there exists $T_0:= T_0(n,s,\alpha) >0$ such that 
\begin{equation}\label{Formula M-P: sup Phi in terms of Psi and T0}
\sup\limits_{t \ge 0} \Phi(t v_\eps) = d (\Psi (v_\eps) )^{\frac{\crits}{\crits-2}} - \frac{T_0^q}{q} K_\eps + o(K_\eps),
\end{equation} 
where 
$$ \Psi(v_\eps) = \frac{\|| v_\eps \||^2 - \lambda \|v_\eps\|^2_{L^2(\Omega)}}{\left(\io \frac{|v_\eps|^{\crits}}{|x|^s} dx \right)^{\frac{2}{\crits}} }=\frac{I_\eps}{J_\eps^{\frac{2}{\crits}}}.$$
The proof of this claim goes exactly as \textbf{Step II} in  \cite[Proposition 3]{Jaber}. We omit it here.\\
Let us now compute $ \Psi(v_\eps).$ It follows from  (\ref{Estimate:I_epsilon}) and (\ref{Estimate:J_epsilon}) that there exist positive constants $c_6,c_7,c_8$ such that  
\begin{align}\label{Formula M-P: Psi(u eps)}
\bs
 \Psi(v_\eps)  =  \mur  & \Bigg( 1 + \left\{\begin{array}{rl}
\displaystyle  -c_6 \lambda \eps^{\alpha} + o(\eps^{\alpha})\,\,\,\,\,\,\,\,\,\,\,\,\,\,\,\,\,\,\,\,\,\,\,\,\,\,\,\,\,\,\,\,\,\,\,\,\,\,\,\,\,\,\,\,\,\,\,\,\,\,\, \,\,\,\,\,\,\,\,\,\,\,\,   & \text{if }  0 \le \gamma < \gamma_{crit}(\alpha) \vspace{0.1cm}\\
-c_7\lambda \eps^{\alpha} \ln \eps^{-1} + o(\eps^{\alpha} \ln \eps^{-1})\,\,\,\,\,\,\,\,\,\,\,\,\,\,\,\, \,\,\,\,\,\,\,\,\,\,\,\, \,\,\,\, & \text{if } \gamma = \gamma_{crit}(\alpha) \\
- c_8 m^\alpha_{\gamma,\lambda} \eps^{\bp-\bm} + o(\eps^{\bp-\bm})   & \text{if }  \gamma_{crit}(\alpha)<\gamma  <\gamma_H(\alpha)
\end{array}\right . \Bigg).
 \es
 \end{align}
We are now going  to estimate $\sup\limits_{t \ge 0}\Phi(v_\eps).$ This will be done again by considering two cases: \\

\textit{Case 1:} $0 \le \gamma \le \gamma_{crit}(\alpha).$ In this case, plugging  (\ref{Formula M-P: Estimate for K_eps : non-critical}) and (\ref{Formula M-P: Psi(u eps)}) into (\ref{Formula M-P: sup Phi in terms of Psi and T0}) implies that there exist constants $c_9,c_{10},c_{11},c_{12}>0$ such that  
\begin{align*}
\sup\limits_{t \ge 0}\Phi(v_\eps) &= \frac{\alpha-s}{2(n-s)}  \mur^{\frac{n-s}{\alpha-s}} - c_9 \lambda \eps^\alpha  - c_{10}  h(0) \eps^{n - q\frac{n - \alpha}{2}} + o(\eps^\alpha)+ o(\eps^{n- q\frac{n - \alpha}{2}}), 
\end{align*}
when  $0 \le \gamma < \gamma_{crit}(\alpha),$ and 

\begin{align*}
\sup\limits_{t \ge 0}\Phi(v_\eps) &= \frac{\alpha-s}{2(n-s)}  \mur^{\frac{n-s}{\alpha-s}} - c_{11} \lambda \eps^\alpha \ln(\eps^{-1})  - c_{12}  h(0) \eps^{n - q\frac{n - \alpha}{2}} + o(\eps^\alpha \ln(\eps^{-1}))+ o(\eps^{n- q\frac{n - \alpha}{2}}),
\end{align*}
when $ \gamma =\gamma_{crit}(\alpha).$\\
 Recall that $\alpha > n - q\frac{n - \alpha}{2},$ since $q>2.$ This implies that

$$o(\eps^\alpha) + o(\eps^{n- q\frac{n - \alpha}{2}}) = o(\eps^{n- q\frac{n - \alpha}{2}}).  $$  
Thus, there exist a positive constant $\tau_1$  such that, for every $ 0 \le  \gamma  \le \gamma_{crit}(\alpha),$ we have  
\begin{equation*}
\sup\limits_{t \ge 0}\Phi(v_\eps) = \frac{\alpha-s}{2(n-s)}  \mur^{\frac{n-s}{\alpha-s}}  - \tau_1  h(0) \eps^{n - q\frac{n - \alpha}{2}} + o(\eps^{n- q\frac{n - \alpha}{2}}), 
\end{equation*}
 
\textit{Case 2:} $ \gamma_{crit}(\alpha) < \gamma  < \gamma_H(\alpha).$ The critical case needs a careful analysis as  a new phenomena happens in this situation. We shall show that  there is a competition between the geometry of the domain, \textit{the mass} (i.e., $ \eps^{\bp-\bm} m^\alpha_{\gamma,\lambda}),$ and the non-linear perturbation (i.e.,  $\eps^{n - q\frac{n - \alpha}{2}}  h(0) ).$ Indeed, it follows  from plugging  (\ref{Formula M-P: Estimate for K_eps : citical }) and (\ref{Formula M-P: Psi(u eps)}) into (\ref{Formula M-P: sup Phi in terms of Psi and T0}) that there exist constants $c_{12},c_{13},c_{14}>0 $ such that   
\begin{align*}
\bs
\sup\limits_{t \ge 0}\Phi(t w_\eps) = \Upsilon & - c_{12} m^\alpha_{\gamma,\lambda} \eps^{\bp-\bm} + o(\eps^{\bp-\bm})\\
& \quad  + \left\{\begin{array}{rl}
\displaystyle   - c_{13}  h(0) \eps^{n - q\frac{n - \alpha}{2}} + o(\eps^{n- q\frac{n - \alpha}{2}})\,\,\,\,\,\,\,\,\,\,\,\,\,\,\,\,\,\,\,\,\,\,\,& \text{if }  q > q_{crit} \vspace{0.1cm}\\
- c_{14}  h(0) \eps^{\bp-\bm} + o(\eps^{\bp-\bm}) & \text{if } q= q_{crit} \\  o(\eps^{\bp-\bm}) \,\,\,\,\,\,\,\,\,\,\,\,\,\,\,\,\,\,\,\,\,\,\,\,\,\,\,\,\,\,\,\,\,\,\,\,\,\,\,\,\,\,\,\,\,\,\,\,\,\,\,\,\,\,\,\,\,\,\,\,\,\,\,\,& \text{if } q <  q_{crit}.
\end{array}\right. 
\es
\end{align*}
Following our analysis in the critical case of estimating $K_\eps,$ one can then summarize the competition results as follows.
\begin{center}

\begin{tabular}{ | m{3.3cm} | m{1.6 cm}| m{3.4cm} | m{1.6cm}|  m{0.7cm}|}
\hline
 Competitive Terms & $q > q_{crit}$ & \qquad $q = q_{crit}$ & $q < q_{crit}$.\\
\hline
\ \quad  $\eps^{n - q\frac{n - \alpha}{2}}  h(0)$ & Dominate &  Equally Dominate & \ \quad $\times$ \\ 
\hline
\ \quad $\eps^{\bp-\bm} m^\alpha_{\gamma,\lambda}$& \ \quad $\times$ &  Equally Dominate & Dominate \\ 
\hline
\end{tabular}
\end{center}
Therefore, we finally deduce that there exists $\tau_l >0 $ for $l = 2,..,5$ such that 
\begin{align*}
\bs
\sup\limits_{t \ge 0}\Phi(t w_\eps) &= \frac{\alpha-s}{2(n-s)}  \mur^{\frac{n-s}{\alpha-s}} \\ 
& + \left\{\begin{array}{rl}
\displaystyle   - \tau_2  h(0) \eps^{n - q\frac{n - \alpha}{2}} + o(\eps^{n- q\frac{n - \alpha}{2}}) \,\,\,\,\,\,\,\,\,\,\,\,\,\,\,\,\,\,\,\,\,\,\,\,\,\,\,\,\,\, \,\,\,\,\, \,\,\,\,\, \,\,\,\,\, \,\,\,\,\, \,\, & \text{if }  q > q_{crit} \vspace{0.1cm}\\
- (\tau_3  h(0) + \tau_4 m^\alpha_{\gamma,\lambda}) \eps^{\bp-\bm} + o(\eps^{\bp-\bm})  & \text{if } q= q_{crit} \\ - \tau_5  m^\alpha_{\gamma,\lambda} \eps^{\bp-\bm} + o(\eps^{\bp-\bm})\,\,\,\,\,\,\,\,\,\,\,\,\,\,\,\,\,\,\,\,\,\,\,\,\,\,\,\,\, & \text{if } q <  q_{crit}.
\end{array}\right. 
\es
\end{align*}
This complete the proof of Proposition \ref{Prop M-P: estimate at the test functions}.

\subsection{Proof of Theorem \ref{Thm M-P: Main result }: The non-local case  }\label{Section: proof of main Thm M-P } 

Assume that $0 < \alpha <2.$ Theorem \ref{Thm M-P: Main result } is now a direct consequence of Theorem \ref{Thm M-P: main condition of existence} and Proposition \ref{Prop M-P: estimate at the test functions}.


\section{ The local case}\label{Sec:local case}
\medskip\noindent In this section, our aim is to prove the existence of positive solutions for (\ref{Main problem: M-P solutions-Local}), and show that Theorem \ref{Thm M-P: Main result } still holds true when $\alpha=2.$ To this end, we use the results by Ghoussoub-Robert \cite{Ghoussoub-Robert-2015}  in the local setting (i.e., $\alpha =2$), and follow the proofs of Theorem \ref{Thm M-P: main condition of existence} and  Proposition \ref{Prop M-P: estimate at the test functions} to obtain the general existence condition (i.e., condition (\ref{Condition:General condition})  when $\alpha =2$)  and the desired estimates corresponding to (\ref{Formula M-P: Phi(u eps): NoN-critical}) and (\ref{Formula M-P: Phi(u eps): Critical}).

We consider the case when $\alpha =2,$ that is when the operator $\fo$ boils down to the well-known Laplacian  operator $- \Delta.$ Problem (\ref{Main problem: M-P solutions})  can therefore be written as 

\begin{equation}\label{Main problem: M-P solutions-Local:Section}
\left\{\begin{array}{rl}
\displaystyle {-}{ \Delta} u- \gamma \frac{u}{|x|^2} - \lambda u= {\frac{u^{\critsl-1}}{|x|^s}}+ h(x) u^{q-1} & \text{in }  {\Omega}\vspace{0.1cm}\\
u\geq 0  \,\,\,\,\,\,\,\,\,\,\,\,\,\,\,\,\,\,\,\,\,\,\,\,\,\,\,\,\,\,\,\,\,\,\,\,\,\,\,\,\,\,\,  & \text{in } \Omega ,\\
 u=0 \,\,\,\,\,\,\,\,\,\,\,\,\,\,\,\,\,\,\,\,\,\,\,\,\,\,\,\,\,\,\,\,\,\,\,\,\,\,\,\,\,\,  & \text{in } \partial \Omega,
\end{array}\right.
\end{equation}

where  $s \in [0,2),$ $ 2_2^*(s):=\frac{2(n-s)}{n-2},$ $ \gamma < \gamma_H(2)= \frac{(n-2)^2}{4},$ and  $u$ belongs to the space $H_0^{1}(\Omega)$ which is the completion of $C_c^\infty(\Omega)$ with respect to the norm $$ \|u\|_{H_0^{1}(\Omega)}^2 = \int_{\Omega} |\nabla u|^2dx.$$
Moreover, the norm  
$$\|| u\||_2 = \left[ \int_{\Omega} \Big(  |\nabla u|^2dx - \gamma \frac{u^2}{|x|^2} \Big) \ dx \right]^{\frac{1}{2}} $$
is well-defined on $H_0^{1}(\Omega),$ and is equivalent to $\| \ . 
 \ \|_{H_0^{1}(\Omega)}.$ The corresponding functional to problem (\ref{Main problem: M-P solutions-Local})  is  
 \begin{equation*}
\Phi_2(u)= \frac{1}{2} \|| u\||_2^2 -\frac{\lambda}{2} \io u^2 dx  -\frac{1}{2_{2}^*(s)}\int_{\Omega} \frac{u_+^{2_{2}^*(s)}}{|x|^{s}} dx -\frac{1}{q} \int_{\Omega} h u_+^q dx.
\end{equation*}

We point out that the results proved in Sections 
$2, 3$ and $4$ in \cite{Ghoussoub-Robert-2015}  provide us the desired tools to show that Theorem \ref{Thm M-P: main condition of existence} and  Proposition \ref{Prop M-P: estimate at the test functions} still hold in the local case (i.e., when $\alpha=2$). One can indeed use these results, and follow the proofs of Theorem \ref{Thm M-P: main condition of existence} and  Proposition \ref{Prop M-P: estimate at the test functions} to establish the following.

\begin{theorem}\label{Thm M-P: main condition of existence:Local}
Let $\Omega$ be a smooth bounded domain in $\R^n (n > 2)$ such that $0 \in \Omega,$ and let $\critsl:= \frac{2(n-s)}{n-2}, $  $0 \le  s < 2,$ $-\infty < \lambda < \lambda_1(\LOl),$ and $ \gamma < \gamma_H(2)= \frac{(n-2)^2}{4}.$ We consider $2 < q <2^*_2,$ $h \in C^0(\overline{\Omega})$ and $h \ge0.$ . We also assume that there exists $w \in \homegal$, $w \not \equiv 0$ and $w \ge 0 $ such that 
\begin{equation}\label{Formula M-P: general cond. for existence}
\sup\limits_{t \ge 0} \Phi_2(t w) <  \frac{2-s}{2(n-s)} \murl^{\frac{n-s}{2-s}}.  
\end{equation} 
Then, problem (\ref{Main problem: M-P solutions-Local}) has a non-negative  solution in $\homegal$.
\end{theorem}

\begin{remark}
Recall that the best constant in the Hardy-Sobolev inequalities is defined as 

$$\murl = \inf\limits_{u \in H^1_0(\R^n) \setminus \{0\}}\frac{\int_{\R^n} (  |\nabla u|^2dx - \gamma \frac{u^2}{|x|^2} ) \ dx}{\left(  \int_{\R^n} \frac{u^{2_{2}^*(s)}}{|x|^{s}} dx   \right)^{\frac{2}{2^*(s)}}},$$ 
where $H^1_0(\R^n)$  is the completion of $C_c^\infty(\R^n)$ with respect to the norm $ \|u\|_{H_0^{1}(\R^n)}^2 = \int_{\R^n} |\nabla u|^2dx.$

\end{remark}

\begin{proposition}\label{Prop M-P: estimate at the test functions-local}
Let  $\Upsilon_2:= \frac{2-s}{2(n-s)}  \mu_{\gamma,s,2}(\rn)^{\frac{n-s}{2-s}}.$ Then, there exists $\tau_{l,2} >0 $ for $l = 1,..,5$ such that\\

1) If $L_{\gamma,2}$ is not critical, then

\begin{equation}\label{Formula M-P: Phi(u eps): NoN-critical}
\sup\limits_{t \ge 0}\Phi(t v_{\eps,2}) = \Upsilon_2  - \tau_{1,2}  h(0) \eps^{n - q\frac{n - 2}{2}} + o(\eps^{n- q\frac{n - 2}{2}}); 
\end{equation}

2) If $L_{\gamma,2}$ is critical,  then
\begin{align}\label{Formula M-P: Phi(u eps): Critical}
\bs
\sup\limits_{t \ge 0}\Phi(t v_{\eps,2}) &= \Upsilon_2 + \left\{\begin{array}{rl}
\displaystyle   - \tau_{2,2}  h(0) \eps^{n - q\frac{n - 2}{2}} + o(\eps^{n- q\frac{n - 2}{2}}) \,\,\,\,\,\,\,\,\,\,\,\,\,\,\,\,\,\,\,\,\,\,\,\,\,\,\,\,\,\,\,\,\,\,\,\,\,\,\,\,\,\,\,\,\,\,\,\,\,\,\,\,\,\,\,\,\,\,\,\,\,\, & \text{if }  q > q_{crit}(2) \vspace{0.1cm}\\
- (\tau_{3,2} h(0) + \tau_{4,2} m^\lambda_{\gamma,2}(\Omega)) \eps^{\beta_+(2)-\beta_-(2)} + o(\eps^{\beta_+(2)-\beta_-(2)})  & \text{if } q= q_{crit}(2) \\ - \tau_{5,2} m^\lambda_{\gamma,2}(\Omega) \eps^{\beta_+(2)-\beta_-(2)} + o(\eps^{\beta_+(2)-\beta_-(2)})\,\,\,\,\,\,\,\,\,\,\,\,\,\,\,\,\,\,\,\,\,\,\,\,\,\,\,\,\, \,\,\,\,& \text{if } q <  q_{crit}(2),
\end{array}\right. 
\es
\end{align}
where $q_{crit} (2) := 2^*_2 - 2 \frac{\bpl-\bml}{n-2} \in  (2,2^*_2).$  
 
\end{proposition}

\subsection{Proof of Theorem \ref{Thm M-P: Main result }: The local case  }\label{Section: proof of main Thm M-P } 

Assume that $\alpha = 2.$ Theorem \ref{Thm M-P: Main result } is  a direct consequence of Theorem \ref{Thm M-P: main condition of existence:Local} and Proposition \ref{Prop M-P: estimate at the test functions-local}.

\end{document}